\documentclass[12pt]{amsart}
\usepackage{graphicx,latexsym,color,amssymb,amscd}
\usepackage[all]{xy}
\usepackage{amscd,amssymb,epsfig,mathrsfs}
\usepackage{graphics}

\input epsf
\newtheorem{theorem}{Theorem}[section]
\newtheorem{lemma}[theorem]{Lemma}

\newtheorem{corollary}[theorem]{Corollary}
\theoremstyle{definition}
\newtheorem{definition}[theorem]{Definition}

\theoremstyle{remark}

\numberwithin{figure}{section}
\numberwithin{table}{section}

\newcommand{\ra}{\rightarrow}
\newcommand{\hra}{\hookrightarrow}

\providecommand{\abs}[1]{\left\lvert#1\right\rvert}
\providecommand{\len}[1]{\lVert#1\rVert}%{\dagger#1\dagger}

\providecommand{\norm}[1]{\lVert#1\rVert}

\newcommand\mc{\mathcal}

\newcommand{\Teich}{Teichm\"uller }
\newcommand{\Poin}{Poincar\'e }

\def\ra{{\rightarrow}}

\def\bZ{{\mathbb Z}}

\def\bR{{\mathbb R}}

\def\S1g{{\Sigma_{g,1}}}
\def\T1g{{\mathcal T}_{g,1}}
\def\M1g{{MC}_{g,1}}

\def\I1g{{\mathcal I}_{g,1}}

\newcommand\Diff{\operatorname{Diff}}

\addtocounter{MaxMatrixCols}{5}

\begin{document}

\title[Mapping Class Factorization]
{Mapping Class Factorization via Fatgraph Nielsen Reduction %Factoring a mapping class given by its action on the fundamental group of a surface
}

\author{Alex James Bene}
\address{Department of Mathematics\\
University of Southern California\\
Los Angeles, CA 90089\\
USA\\}
\email{bene{\char'100}math.usc.edu}


\thanks{The author would like to thank Robert Penner for helpful discussions and feedback on earlier versions of this paper.}

%*%*%*%*%*%*%*
%%*%*%*%*%*%*%*
%%%*%*%*%*%*%*%*
\begin{abstract}

The mapping class group of a genus $g$ surface $\Sigma_{g,1}$ with one boundary component is known to have a simple yet infinite presentation with generators given by elementary moves called  Whitehead moves on so-called marked bordered fatgraphs.  
In this paper, we introduce  an algorithm  called ``fatgraph Nielsen reduction'' which, from the action of a mapping class  $\varphi\in MC_{g,1}$ of  $\Sigma_{g,1}$  on the fundamental group $\pi_1(\Sigma_{g,1})$ of $\Sigma_{g,1}$, determines a sequence of Whitehead moves representing $\varphi$ beginning at any choice of marked bordered fatgraph.   As a consequence, this leads to an algorithm which factors any mapping class  given by its action on $\pi(\Sigma_{g,1})$  in terms of a certain generating set  for $MC_{g,1}$.  
\end{abstract}

\maketitle

%%%%%%%%%%%%%%%%%%%%%%%
\section{Introduction}

The combinatorial description of the mapping class group of a surface $\Sigma$ in terms of ideal triangulations of $\Sigma$ has a long history, going back to  Whitehead who proved that any two ideal triangulations are related by a sequence of elementary  ``diagonal exchange'' moves.  Later, inspired by the geometric insights of Thurston and Mumford, this combinatorial description became more firmly established in the hyperbolic setting by Penner's decorated \Teich space \cite{penner} and in the conformal setting by Harer's utilization of Strebel's results  on quadratic differentials \cite{harer,strebel}.  In these settings, the \Poin dual viewpoint gained  prevalence, where diagonal exchanges on ideal triangulations gave way to  elementary moves called Whitehead moves on marked fatgraphs:  vertex-oriented graphs  embedded in $\Sigma$ as a spine.  In particular, any mapping class of $\Sigma$ can be represented by a sequence of Whithead moves, and this sequence is unique up to certain well-known relations.

 The mapping class group $MC_{g,1}$ of a surface $\Sigma_{g,1}$ with one boundary component has  particularly nice properties. 
 Algebraically, it is classically known that $MC_{g,1}$ is a subgroup of the automorphism of a free group via its action on the fundamental group of $\Sigma_{g,1}$.  Combinatorially, the mapping class group $MC_{g,1}$  admits another related  combinatorial description in    terms of elementary moves called chord slides on a special type of \emph{linear}  fatgraph \cite{abp,bene} which have  coincidentally been studied extensively  in other contexts under the name of \emph{linear chord diagrams}.  We will choose to use  this later terminology in this paper.  Moreover, there exists a precise chordslide--Whitehead move correspondence which relates these two types of elementary moves.  
      This ``linear" variation of the theme appears to  have  some advantages.  
 In particular,  every marked linear chord diagram canonically determines a set of generators for the free group $\pi:=\pi_1(\Sigma_{g,1})$ of $\Sigma_{g,}$.  We call a generating set arising in this fashion a combinatorial generating set, or simply  CG set.  Each  CG  set satisfies certain constraints dictated by the form of the linear chord diagram from which it arises, a special case of which is the well known condition that a standard set of symplectic generators $\{\alpha_i,\beta_i\}$ satisfies the  relation $\prod_{i=1}^{g} [\alpha_i,\beta_i]=\partial\Sigma_{g,1}$.

While every mapping class $\varphi\in MC_{g,1}$ has a description in terms of elementary moves on fatgraphs, algorithms currently available for determining such a sequence have the disadvantage that they depend on  resolving  intersections of arcs or closed curves in the surface $\Sigma_{g,1}$ \cite{mosher,Penner03}.  For many reasons,  it would be desirable to construct such a sequence from purely algebraic information, such as the action of $\varphi$ on the fundamental group $\pi$ of $\Sigma_{g,1}$.  In this paper, we present just such an  algorithm,  
 which we call   \emph{fatgraph Nielsen reduction}.  
 
   To define this algorithm, we introduce an \emph{energy  function} $\len{\cdot}$ on  CG  sets which is  derived from a lexicographical ordering on  $\pi$  extending the usual word length function with respect to some set of generators for $\pi$.   In fact, each marked linear chord diagram $G\hra \Sigma_{g,1}$, which can be considered as a choice of ``basepoint in decorated \Teich space'' (see Section \ref{sect:fatgraphs} or \cite{abp}), determines its own energy function.  By the correspondence  between marked linear chord diagrams and CG sets, we can equivalently  consider this  energy $\len{\cdot}$  as a function on marked linear chord diagrams, and the main result of this paper can be stated as follows:

\begin{theorem}\label{thm:fmain}
Let $\len{\cdot}$ denote the energy function with respect to a fixed ``basepoint''   $G\hra \Sigma_{g,1}$.  Given any  marked linear chord diagram $G_0$, there exists a sequence of energy  decreasing chord slides
\[
G_0\ra G_1\ra \dotsm \ra G_k, \quad \norm{G_i} < \norm{G_{i-1}}
\]
with $G_k=G$.
\end{theorem}
The proof of this theorem 
%The proof that a  energy decreasing slide always exists for $G_0\neq G$
 relies on the combinatorics  of linear chord diagrams and elementary cancellation theory. 

Since there is a canonical sequence of Whitehead moves which ``linearizes'' a given bordered fatgraph (given by the greedy algorithm of \cite{abp}) and every chord slide of a linear chord diagram  can be described in terms of Whitehead moves, we obtain the following
\begin{corollary}
There exists an explicit algorithm which from the action of a mapping class $\varphi\in MC_{g,1}$ on $\pi$ determines a sequence of Whitehead moves  representing $\varphi$,  beginning at any ``basepoint'' marked bordered fatgraph $G\hra \Sigma_{g,1}$.  % any element $\varphi$ of $MC_{g,1}$ from the  action of $\varphi$ on $\pi$.  
\end{corollary}

Moreover, since there is a map from Whitehead moves to mapping classes in $MC_{g,1}$ (see \cite{abp}) with finite image $\mathfrak{G}_g$, we  have the following
\begin{corollary}
Given any generating set $\mathfrak{S}$ for $MC_{g,1}$ described in terms of $\mathfrak{G}_g$, there is an algorithm  for decomposing any mapping class $\varphi\in MC_{g,1}$  into a product of generators in $\mathfrak{S}$.
\end{corollary}

%%%%%%%%%%%%%%%%%%%%%%%
\section{Combinatorial Generating Sets}

Let $\Sigma_{g,1}$ be a genus $g$ surface with one boundary component, and let $\pi$ denote its fundamental group $\pi:=\pi_1(\Sigma_{g,1},p)$ with respect to a basepoint $p\in \partial\Sigma_{g,1}$.  It is well known that $\pi$ is isomorphic to a free group on $2g$ generators, and an explicit isomorphism is equivalent to a choice of an ordered set of generators for $\pi$.  For a given set of $2g$ generators $\mc{X}=\{x_i\}_{i=1}^{2g}$, we define the corresponding set of \emph{letters} $\{x_i, \bar x_i\}_{i=1}^{2g}$ to be the set containing each element $x\in\mc{X}$ and its inverse $\bar x$.  

We  define a \emph{standard symplectic set of generators} for $\pi$ to be one of the form $\mc{S}=\{\alpha_i,\beta_i\}_{i=1}^{g}$ such that the element of  $\pi$ corresponding to the boundary $\partial\Sigma_{g,1}$ is represented  by the word $\partial\Sigma_{g,1}=\prod_{i=1}^g[\alpha_i,\beta_i]$, where the bracket $[x,y]$ denotes the commutator $xy\bar x \bar y$.  Again, we let   $\bar x$ denote the inverse of an element $x$.  Note that  each letter of the generating set  is used once in this product.

More generally, we make the 
\begin{definition}
An ordered set of generators $\mc{X}$ for $\pi$ is a \emph{combinatorial generating set}, or CG set, if the  boundary element  $\partial\Sigma_{g,1}\in \pi$ can be written as a reduced word  using each letter of $\mc{X}$ exactly once.  Two CG sets are equivalent if their corresponding sets of letters are the same. 
\end{definition}
 We remark that not all generating sets are CG sets, and not all CG sets are (equivalent to) standard symplectic generating sets. 
Topologically, an equivalence class of CG sets corresponds to (an isotopy class of)  a collection of $2g$ arcs based at $p$ which decompose the surface $\Sigma_{g,1}$ into a $(4g+1)$-gon, and a CG set itself corresponds to a labeling and choice of orientation for each such arc.  However, we will find it more useful to consider the picture which is \Poin dual to this, which will involve linear chord diagrams.  
 
The mapping class group $MC_{g,1}$ of the surface $\Sigma_{g,1}$ is typically defined as $\pi_0(\Diff(\Sigma_{g,1},\partial\Sigma_{g,1}))$,   the group of components of  the group  of self-diffeomorphisms of $\Sigma_{g,1}$ which fix the boundary pointwise.  By a classical result usually attributed to Dehn and Nielsen (see \cite{Zieschang}),  $MC_{g,1}$ has an equivalent algebraic definition as 
 the group of automorphisms of $\pi$ which fix the word representing the boundary $\partial\Sigma_{g,1}$.  Equivalently, but more in line with our viewpoint, we can define the mapping class group $MC_{g,1}$ in the following way:  mapping classes in $MC_{g,1}$ are exactly the  automorphisms of $\pi$ which take standard symplectic generating sets to standard symplectic generating sets, and  there is a (non-canonical) one-to-one correspondence between (equivalence classes of) standard symplectic generating sets and elements of $MC_{g,1}$.  

We now introduce  the chord slide groupoid, which  is a subgroupoid of the Ptolemy groupoid of $\Sigma_{g,1}$ (see section \ref{sect:fatgraphs} or \cite{bkp}) and  should be thought of as a groupoid ``containing'' the mapping class group $MC_{g,1}$.  Recall that a groupoid can either be described as a set with a partial composition operation with inverses, or as a category all of whose morphisms are isomorphisms.  
\begin{definition}[cf. \ref{thm:bene}]
Fix a CG set $\mc{X}_i$ in each equivalence class, and let $\mathfrak{X}=\{\mc{X}_i\}$ denote the set of all representatives.  
The \emph{chord slide groupoid} is defined as the category whose objects are copies of $\pi$, precisely one  for every $\mc{X}_i\in\mathfrak{X}$, and whose  morphisms are  the homomorphisms $\psi_{i,j}\colon \pi\ra \pi$ which take one CG set to another $\psi_{i,j}\colon \mc{X}_i\mapsto\mc{X}_j$.  
\end{definition}
Note that by the Hopfian property of $\pi$, the $\psi_{i,j}$ are  necessarily isomorphisms.  
Also note that there is a unique morphism between each pair of  objects, so that this groupoid is equivalent to a trivial groupoid (in the category-theoretic sense), and that up to isomorphism, this definition does not depend on the choice of representatives.

\section{Linear chord diagrams}\label{sect:linfatgraphs}

A linear chord diagram  is a combinatorial object  best described as a graph $G$ immersed in the plane:    the \emph{core} of $G$ consists of the (connected subgraph consisting of) edges  of $G$ embedded in  the real line, and  the remaining edges, the \emph{chords} of $G$, are immersed arcs in the upper half plane with their endpoints  attached at distinct points of the core.  We call the attaching points the \emph{chord ends} of $G$, and require that they correspond to integer points of the real line.  
(See \cite{bene} for a more thorough description.)  For a given chord end $v$, we let $\bar v$ denote the opposite end of the chord attached at $v$.

There is a natural linear chord diagram associated to every CG set, but before we describe this association, we first develop some notation concerning  CG sets.  
Given a CG set $\mc{X}=\{x_i\}_{i=1}^{2g}$ for $\pi$, let 
\[
\mc{C}_{\mc{X}}=\{c_j\}_{j=1}^{4g}, \quad c_j=x_{i_j}^{\varepsilon_j}
\]
denote the (unique) ordered  set of letters corresponding to $\mc{X}$, ordered such that 
\[
\prod_{j=1}^{4g} c_j =\prod_{j=1}^{4g} x_{i_j}^{\varepsilon_j}= \overline{\partial \Sigma_{g,1}}
\]
with $\varepsilon_i=\pm 1$.  (The reason for the appearance of $\overline{\partial \Sigma_{g,1}}$ rather than $\partial \Sigma_{g,1}$ in the above formula will be discussed in section \ref{sect:fatgraphs}.)  
 For example, for a standard symplectic CG set $\{\alpha_i,\beta_i\}_{i=1}^{g}$, the corresponding letters are  $c_1=\beta_g$, $c_2=\alpha_g$, $c_3=\bar \beta_g$, etc. 
 Finally, we define integers $\ell_i$ and $r_i$ for $i=1,\ldots, 2g$ by 
\[
c_{\ell_i} =x_i^{\pm 1}, \quad c_{r_i}=x_i^{\pm 1}, \quad \ell_i<r_i.  
\]
 %See Figure \ref{fig:symplectic}. 
Note that the set $\{c_{\ell_i}\}_{i=1}^{2g}$ of ``left ends'' gives a canonical representative for the  equivalence class containing $\mc{X}$, as does the set $\{c_{r_i}\}_{i=1}^{2g}$ of ``right ends''.  

 \begin{figure}[h!]
\input{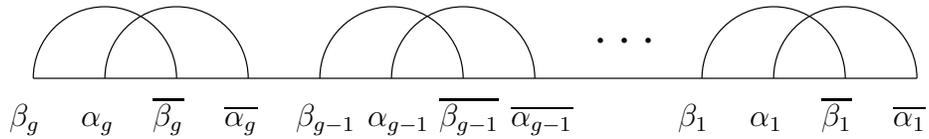}
  \caption{A syplectic CG set and linear chord diagram }\label{fig:symplectic}
 \end{figure}

Given a  CG set $\mc{X}$ for $\pi$, we construct a linear chord diagram, denoted   $G_{\mc{X}}$,  as follows. The core of $G_{\mc{X}}$ is taken to be the interval $[1,4g]$ of the real line, the chords of $G_{\mc{X}}$ correspond to the generators of $\mc{X}$, and the  ends  of a chord corresponding to $x_i\in \mc{X}$ are prescribed  by $\ell_i$ and $r_i$.  Using this correspondence, we will often abuse notation slightly and  refer to the generators $x_i$ as \emph{chords} and the letters $c_j$ as \emph{chord ends}.    See Figure \ref{fig:symplectic} for the linear chord diagram associated to a standard symplectic set of generators.   By the uniqueness of reduced words in $\pi$  (see \cite{abp}), the  linear chord diagram $G_\mc{X}$ is uniquely determined by the CG set $\mc{X}$.   We call a linear chord diagram arising from a CG set for $\pi$ a \emph{genus $g$} chord diagram.   By a \emph{marked linear chord diagram}, we shall mean a linear chord diagram $G_\mc{X}$ together with a labeling of its chord ends by the letters $\mc{C}_{\mc{X}}$  of the CG set $\mc{X}$.  By abuse of notation, we shall denote this marked linear chord diagram also by  $G_{\mc{X}}$.

We now describe the elementary move on linear chord diagrams, the \emph{chord slide}.  Let $v$ be a chord end which is not farthest to the right of a linear chord diagram $G$.  The chord slide %$R(v)\colon G\ra G'$
 of $v$ to the right results in a new linear chord diagram $G'$ obtained by  ``sliding $v$ over''  the chord to its right, thus relocating $v$  to a new position along the core (and isotoping chord ends so that they lie at distinct integer points).   The chord slide to the left %$L(v)$
  is completely analogous. See Figure \ref{fig:sl}.
 
We can extend chord slides so that they act on the set $\mathfrak{X}$ of (representatives of) CG sets in a natural way.  This action is most easily described pictorially by the action of a chord slide on the chord ends associated to a CG set.  There are  four possible types of chord slides, and we illustrate the action of each  in Figure \ref{fig:sl}, where all chord ends not depicted in the figure are understood to be fixed by the chord slide.  While not explicitly needed, we note that the resulting effect on a CG set $\mc{X}$ takes the following form up to permutation, which is a (composition of) Nielsen transformation(s) for the free group $\pi$: 
 \[
 \begin{split}
 x_j& \mapsto x_j^{\pm 1}, \quad  j\neq i \\
 x_i&\mapsto (x_i^{\pm 1} x_k^{\pm 1})^{\pm 1} , \quad \textrm{for some }k\neq i,
 \end{split}
 \]
 where $x_i\in \mc{X}$ is the generator corresponding to the chord which a chord end is slid over, and the exact form depends on the choice of representative CG sets.

 \begin{figure}[h!]
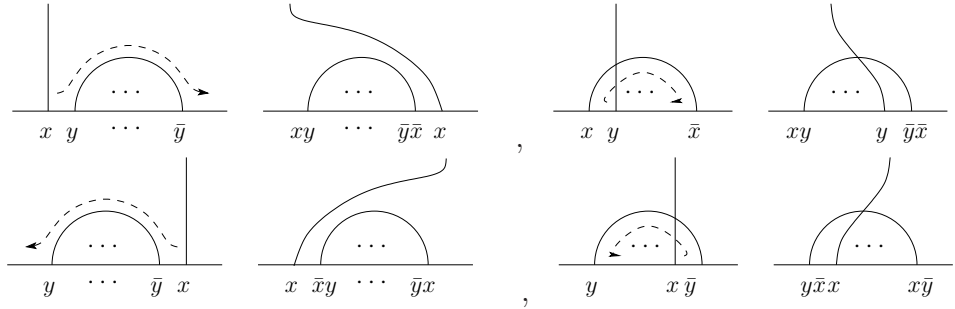

\scalebox{.75}{\input{sl1.pstex_t}}\;\;\; , \;\;\scalebox{.75}{\input{sl2.pstex_t}}\\
\scalebox{.75}{\input{sl3.pstex_t}}\; \;\; , \;\;\scalebox{.75}{\input{sl4.pstex_t}}\\
  \caption{Two  slides to the right and two slides to the left. }\label{fig:sl}
 \end{figure}

Note that under  obvious conditions, chord slides on CG sets can be composed, thus they generate a groupoid.  In fact, by their actions on (equivalence classes of) CG sets, chord slides can be considered as elements (i.e., morphisms) of the chord slide groupoid, and we now recall  the main result of \cite{bene}:

\begin{theorem}[\cite{bene}]\label{thm:bene}
The chord slide groupoid is  generated by chord slides on CG sets, and all relations are compositions of five explicit ones (which we do not list here).  
\end{theorem}

The importance of this theorem is that it  yields  a combinatorial presentation of the mapping class group $MC_{g,1}$.  More precisely, the action of the mapping class group $MC_{g,1}$ on $\pi$ extends to an obvious  action on the set $\mathfrak{X}$ of (equivalence classes) of CG sets, and  letting $\varphi(\mc{X})$ denote the image of $\mc{X}$ under $\varphi\in MC_{g,1}$, the theorem states that $\varphi$ can be represented by a  sequence of chord slides 
\[
G_\mc{X}\ra G_{\mc{X}_1}\ra \dotsm \ra G_{\varphi(\mc{X})}
\]
on marked linear chord diagrams, unique up to some known relations.  
The main result of this paper is an algorithm which will determine such a sequence.

\section{The length and energy  functions}

We now introduce the length and energy functions  that we will use to develop our algorithm.  For the rest of the paper, we assume that the genus $g$ is fixed, and we fix a (not necessarily standard symplectic) CG set $\mc{S}=\{s_i\}_{i=1}^{2g}$ for $\pi$ with corresponding letters $\mc{C}_{\mc{S}}=\{\sigma_j\}_{j=1}^{4g}$ so that  $\prod_{j=1}^{4g} \sigma_j = \overline{\partial\Sigma_{g,1}}$.    

Firstly, we recall   the standard word length function with respect to $\mc{S}$, which we denote by $\abs{\cdot}$.  For  an element $w\in \pi$, we define 
$
\abs{w}=k
$
if $w$ is written as a reduced word $w=\prod_{i=1}^k \sigma_{j_i}$ using $k$ (possibly repeating) letters  of $\mc{C}_{\mc{S}}$.  We also set $\abs{id}=0$.     For any CG set $\mc{X}=\{x_i\}$ with corresponding chord ends $\mc{C}=\{c_j\}$, we define the length of $\mc{X}$ by
\[
\abs{\mc{X}}:=2 \sum_{i=1}^{2g}\abs{x_i}=\sum_{j=1}^{4g}\abs{c_j}.
\]
Note that $\abs{\mc{X}}=4g$ if and only if $\mc{X}$ is equivalent to $\mc{S}$ (i.e., $\mc{C}_{\mc{X}}=\mc{C}_{\mc{S}}$) by the uniqueness of reduced words.  We also define the length function of a marked linear chord diagram by $\abs{G_{\mc{X}}}:=\abs{\mc{X}}$.

We shall also require a refinement of the above word length function, which we call the \emph{energy function}.  This energy function is derived from a  lexicographical ordering for $\pi$, and  we begin by defining the value of the energy function  on the letters in $\mc{C}_{\mc{S}}$.  We set
\[
\len{\sigma_j}=j.
\]
We then define the energy  of a reduced word $w=\prod_{i=1}^k \sigma_{j_i} \in \pi$ by 
\[
\len{w}=\sum_i (4g+1)^{i-1} \len{\sigma_{j_i}}.
\]
Note that this function defines a well-ordering on elements of $\pi$.  Moreover, this extends the word length function since  by the definition, we have
\[
(4g+1)^{\abs{w}-1}\leq  \len{w} <(4g+1)^{\abs{w}}
\]
so that 
\[
\abs{w}< \abs{v}, \quad  \textrm{implies} \quad \len{w}<\len{v}.
\]
Finally, for any CG set $\mc{X}=\{x_i\}$ with corresponding chord ends $\mc{C}_{\mc{X}}=\{c_j\}$, we define the energy  of $\mc{X}$ to be
\[
\len{\mc{X}}=\sum_{j=1}^{4g} \len{c_j}.
\]
Similarly, we define the energy  of a marked linear chord diagram $G$ by  $\len{G_{\mc{X}}}:=\len{\mc{X}}$.

As a final remark, we state the easy
\begin{lemma}\label{lem:plq}
If $p,q\in \pi$ with $\len{p}<\len{q}$, and $w\in \pi$ is written as a reduced word $w=qx$ with $\abs{q}\leq\abs{x}$, then 
\[
\len{px}+\len{\bar x \bar p} < \len{qx}+\len{\bar x \bar q}.
\]
\end{lemma}
\begin{proof}
Letting $k=\abs{x}$, we have
\[
\len{\bar x \bar p}-\len{\bar x \bar q} <(4g+1)^k,  \quad \textrm{ while }  \quad \len{qx}-\len{px}\geq (4g+1)^k,
\]
and the result follows.
\end{proof}

\section{Some Cancellation Theory}

Recall that we have fixed a ``basepoint'' CG set $\mc{S}=\{s_i\}$ with corresponding letters $\{\sigma_j\}_{j=1}^{4g}$ and  length and energy functions  $\abs{\cdot}$ and $\norm{\cdot}$.  

Now, consider an arbitrary CG set $\mc{X}=\{x_i\}_{i=1}^{2g}$ with corresponding chord ends $\mc{C}=\{c_j\}_{j=1}^{4g}$ and linear chord diagram $G_{\mc{X}}$.  
For a given chord end $c_j\in \mc{C}$ with $j>1$, let $\ell(c_j)$  denote the amount of cancellation of $c_j$ on the left, meaning the number of letters that get cancelled when reducing the word $c_{j-1}c_j$.  More precisely,
\begin{equation}\label{eq:ell}
\ell(c_j)=\frac12\left(\abs{c_{j-1}}+\abs{c_j}-\abs{c_{j-1}c_j}\right).
\end{equation}
Similarly, define $r(c_j)$ to be the number of cancellations on the right, so that $r(c_j)=\ell(c_{j+1})$ with $1\leq j<4g$.  We also formally set $\ell(c_1)=r(c_{4g})=0$.

We begin by stating the obvious
\begin{lemma}\label{lem:halflength}
 If $\ell(c_j)>\abs{c_j}/2$ for some $j>1$,  then  the chord slide of $c_j$ to the left results in a new CG set $\mc{X}'$ with $\abs{\mc{X}'}<\abs{\mc{X}}$.   Similarly, if $r(c_j)>\abs{c_j}/2$ for some $j<4g$, the slide of $c_j$ to the right reduces the length % (therefore also the energy)
  of the CG set. 
\end{lemma}
\begin{proof}
The slide of $c_j$ to the left replaces the chord ends $c_{j-1}$ and $\bar{c}_{j-1}$ with $c_{j-1}c_j$ and $\bar{c_j} \bar{c}_{j-1}$ and leaves the others unchanged (up to permutation).  Thus, 
 if  $\ell(c_j)>\abs{c_j}/2$,  then  $\abs{c_{j-1}c_j}<\abs{c_{j-1}}$  by \eqref{eq:ell}, so the result follows.  The argument for the slide to the right is analogous.  
\end{proof}

If we were able to show that every CG set always had an associated chord end $c_j$ with $\ell(c_j)>\abs{c_j}/2$ or $r(c_j)>\abs{c_j}/2$,  then there would always be a chord slide which reduced the length, and we would  obtain our main result Theorem \ref{thm:fmain}.  However, this is not always the case, and we will have to make use of a more subtle argument using the energy  function $\len{\cdot}$.    Before stating the next  Lemma, we develop some more notation.  

\begin{definition}
We say that a chord end $c_j$ is \emph{balanced} if $\ell(c_j)=r(c_j)=\abs{c_j}/2$.
\end{definition}
 
 Note that if $c_j$ is a balanced chord end corresponding to a CG set $\mc{X}$, then sliding  $c_j$ in either direction does not change the length $\abs{\mc{X}}$ of the CG set.  However, we have the following
 \begin{lemma}\label{lem:balancedreduces}
 Assume that no chord slide reduces the length of a CG set $\mc{X}$.    Then, if $c_j$ is a balanced chord end of  $\mc{C_\mc{X}}$,  either the chord slide of $c_j$ to the left decreases the energy  of $\mc{X}$ or else the slide of $c_j$ to the right does.  
 \end{lemma}
 \begin{proof}
 Since $c_j$ is balanced, it can be written (as a reduced word)
 \[
 c_j=p\bar q
 \]
 with $\abs{p}=\abs{q}$.  Note that necessarily, $p\neq q$.  Assume, without loss of generality, that $\len{p}<\len{q}$.  Then we claim that the slide of $c_j$ to the right reduces the energy  of the CG set.  (In a completely analogous way, the slide to the left would decrease the energy if $\len{q}<\len{p}$.)  To see this, write
 \[ 
c_{j+1}=q w.
 \]
 We know that $\abs{q}\leq\abs{w}$, or else the chord slide of $c_{j+1}$ to the left would decrease the length of $\mc{X}$.  
 
  \begin{figure}[h!]
\input{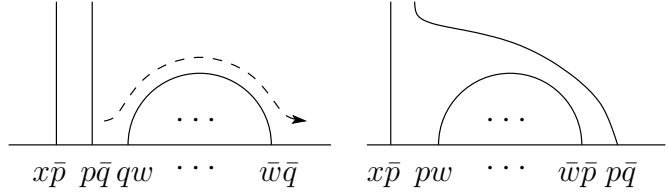}
  \caption{A balanced chord end slid to the right.}\label{fig:pq}
 \end{figure}
 
 After the slide of $c_j$ to the right, the letters of the CG set get altered (up to permutation) by $c_{j+1}=q w$ being replaced by $p w$. See Figure \ref{fig:pq}.  Similarly, the opposite end $\overline{c_{j+1}}=\bar w \bar q$ is replaced by  $ \bar w \bar p$.  The result  then follows from Lemma \ref{lem:plq}.
 \end{proof}

Given a CG set $\mc{X}$ with chord ends $\mc{C}=\{c_j\}$, consider the  unreduced word  $\prod_{j=1}^{4g} c_j $. We know that the word reduces to $\overline{\partial\Sigma_{g,1}}$, so there must exist 
\[
\frac12\left(\sum_j \abs{c_j} -4g \right)
\]
many cancellations.  Note that  a priori  there may be different choices of cancellation ``schemes''  available.  So let us fix one choice of cancellation.  We write $W_j$ for the subword of $c_j$ which is not cancelled under this process, $L_j$ for  the subword cancelled on the left, and $R_j$ for  the subword cancelled on the right.  Thus,
\[
c_j=L_j W_j R_j, \quad \textrm{ and } \quad \prod_{j=1}^{4g} W_j=\overline{\partial\Sigma_{g,1}}
\]
with both $L_1$ and $R_{4g}$ the empty word.
Note that in general, there is no correlation between $\ell(c_j)$ and $\abs{L_j}$. 

\begin{lemma}\label{lem:emptyW}
For a given CG set $\mc{X}$, assume that no chord slide reduces the length $\abs{\mc{X}}$ of $\mc{X}$.  Then if there is a  cancellation scheme $\{c_j=L_jW_jR_j\}$ for the word $\prod_{j=1}^{4g} c_j $ with some $W_j$ the empty word, then some chord end $c_i$ is balanced.  
\end{lemma}
\begin{proof}
Note that it suffices to prove that there is some $i$ with $\ell(c_i)+r(c_i)\geq \abs{c_i}$, as then either $\ell(c_i)=r(c_i)=\abs{c_i}/2$, $\ell(c_i)>\abs{c_i}/2$, or $r(c_i)>\abs{c_i}/2$, and the latter two cases  cannot occur due to our  assumption  that no chord slide reduces the length of $\mc{X}$.

Choose a cancellation scheme $\{c_j=L_jW_jR_j\}$ and assume that  $W_j$ is the empty word for some $1\leq j\leq 4g$.  If $j=1$ or $j=4g$, then the chord slides of $c_1$ to the right or $c_{4g}$ to the left respectively decrease the length of $\mc{X}$, so assume that $1<j<4g$.  Let $I=[m,n]\cap \bZ=\{m, m+1,\ldots, n\}$   be the maximal index interval  containing $j$ such that for all $i\in I$,  $W_{i}$ is  the empty word.  In other words,  the product
\begin{equation}\label{eq:W}
W=R_{m-1}\left(\prod_{i=m}^n c_{i}\right) L_{n+1}=\prod_{i=m}^{n+1} R_{i-1}L_{i}
\end{equation}
reduces to the identity element in $\pi$.  

We now assume that  $\ell(c_{i})+r(c_{i})<\abs{c_{i}}$ for all $1\leq i\leq 4g$ and derive a contradiction. Define a new factorization of each $c_{i}$ with $m\leq i\leq n$ by 
\[
c_{i}=\mc{L}_i \mc{W}_i \mc{R}_i \quad \textrm{with} \quad \abs{\mc{L}_i}=\ell(c_i),\;  \abs{\mc{R}_i}=r(c_i),\; \abs{\mc{W}_i}>0 .
\]
Similarly, write $R_{m-1}=\mc{W}_{m-1}\mc{R}_{m-1}$ and $L_{n+1}=\mc{L}_{n+1}\mc{W}_{n+1}$ with $\abs{\mc{R}_{m-1}}=r(c_{m-1})$ and $\abs{\mc{L}_{n+1}}=\ell(c_{n+1})$.  Note that without our assumption $\ell(c_{i})+r(c_{i})<\abs{c_{i}}$, the existence of such a factorization is not guaranteed.  

With this new factorization, the product $W$ of \eqref{eq:W} can be written as 
\[
W=\mc{W}_{m-1}\mc{R}_{m-1}\left(\prod_{i=m}^n \mc{L}_i \mc{W}_i \mc{R}_i \right) \mc{L}_{n+1}\mc{W}_{n+1},
\] 
which reduces to 
\[
W'=\prod_{i=m-1}^{n+1} \mc{W}_i, \quad \abs{W'}>0
\]
by the definitions of the numbers $\ell(c_i)$ and $r(c_i)$.  However, the definitions of $\ell(c_i)$ and $r(c_i)$ also ensure that $W'$ cannot be reduced any further, while on the other hand, we know that $W'$ reduces to the empty word.  Thus, we have our contradiction and the lemma is proved.
\end{proof}

In the opposite direction, we finish this section with our final lemma.  Before we state it, we introduce some helpful terminology.  Given a word 
$
w=\prod_{i=1}^{k} x_i,
$
we call the subwords 
\[
\prod_{i\leq k/2} x_i \quad \textrm{and} \quad \prod_{i\geq k/2} x_i
\]
 the \emph{left half} and \emph{right half} of $w$ respectively.

\begin{lemma}\label{lem:w}
Assume that there is no length reducing chord slide for a given a CG set $\mc{X}$.   If for some  cancellation scheme $\{c_j=L_jW_jR_j\}$ there is no $i$ with $W_i$  the empty word, then $W_j=\sigma_j$ for all $1\leq j\leq 4g$.  Moreover, in this case, if 
  $c_j=\sigma_j$ for some $j$, then no  chord end other than $c_j$ contains the letter $\sigma_j$.  
\end{lemma}
\begin{proof}
We begin with some observations.
First of all, since 
$
\sum_{j=1}^{4g} \abs{W_j}=\abs{\overline{\partial\Sigma_{g,1}}}=4g,
$
it is immediate that if no $W_j$ is the empty word, then $\abs{W_j}=1$ for all $j$.  In fact, it is clear that we must have $W_j=\sigma_j$ for all $j$ and that $R_j=\overline{L_{j+1}}$ for $1\leq j<4g$.  Second, using Lemma \ref{lem:halflength},  the condition that no length reducing chord slide exists forces $\big\lvert\abs{L_j}-\abs{R_j}\big\rvert\leq \abs{W_j}$ for all $j$, which in our case  gives 
$\big\lvert\abs{L_j}-\abs{R_j}\big\rvert\leq 1$, meaning  each $W_j$ lies ``near the middle'' of $c_j$. As a  consequence,  any letter $\sigma_i\neq\sigma_j$ appears in the left half of $c_j$ if and only if it appears in $L_j$, and it appears in the right half if and only if it appears in $R_j$.  
  
Now assume that for some $j$ we have  $c_j=\sigma_j$ and that  the letter $\sigma_j$ appears in the chord end $c_k$ with $k\neq j$.  Since $\sigma_k\neq\sigma_j$, $\sigma_j$ appears in either $L_k$ or $R_k$, and we assume without loss of generality that it appears in $L_k$, in which case, $k$ must be greater than $1$.  Then $\bar\sigma_j$ must appear in $R_{k-1}$.  Let $c_{k_1}$ be the opposite end of $c_{k-1}$ so that we have $\sigma_j$ appearing in $c_{k_1}$. Note that we cannot have $k_1=j$ since $\abs{c_{k_1}}>1=\abs{c_j}$.  Since $\bar\sigma_j$ lies in the right half of $c_{k-1}$, $\sigma_j$ must lie in the left half of $c_{k_1}$.  Since $W_{k_1}\neq \sigma_j$,  $\sigma_j$ must appear in $L_{k_1}$.  Similarly, we see that $\bar\sigma_j$ appears in $R_{k_1-1}$, thus $\sigma_j$ appears in the left half  of the opposite end of $c_{k_1-1}$.  Continuing in this way, we eventually see that either $\sigma_j$  lies in $L_1$ or $L_j$, or that $\bar\sigma_j$ lies in $R_j$ or $R_{4g}$ (see the last section where we discuss the boundary cycle of a fatgraph), which gives us a contradiction since  each $L_1$, $R_{4g}$, $L_j$, and $R_j$  is the empty word.
\end{proof}

\section{Main Result}

Again, we fix a basepoint $\mc{S}$ and corresponding length functions.  We now present our main theorem: 
\begin{theorem}
Given a CG set $\mc{X}$ not equivalent to $\mc{S}$,  there is some chord slide  which reduces the energy  $\len{\mc{X}}$ of $\mc{X}$.  
\end{theorem}
\begin{proof}
If any chord slide reduces the word length of $\mc{X}$, then we are done, so assume otherwise.  Also, if under any cancellation scheme $\{c_j=L_jW_jR_j\}$ there is some chord end $c_j$ with $W_j$ the empty word, we are done by Lemmas \ref{lem:balancedreduces} and  \ref{lem:emptyW}, so assume that this is not the case.  

  \begin{figure}[h!]
\input{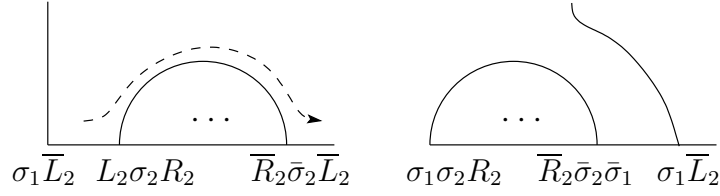}
  \caption{The case $\abs{R_1}=1$. }\label{fig:final}
 \end{figure}

With these assumptions,  the conditions of Lemma \ref{lem:w} are satisfied so that  each $W_j$ is the single letter word $W_j=\sigma_j$.  In this case, we know that $c_1=\sigma_1 R_1$ and $c_2=L_2 \sigma_2 R_2$ with $L_2=\bar R_1$, etc.  Furthermore, we must have $\abs{R_1}\leq 1$ otherwise we could slide $c_1$ to the right and decrease the length of $c_2$. 
  
  Now, if $\abs{R_1}=1$, then  sliding  $c_1$ to the right would have  the effect of changing the first letter of $c_2$ from $ L_2$ to $\sigma_1$,  changing the last letter of   $\bar c_2$  from $\bar L_2$ to $\bar \sigma_1$, and keeping all other chord ends unchanged.  See Figure \ref{fig:final}.  Since $\sigma_1\neq \overline{ R_1}=L_2$, we would then have $1=\len{\sigma_1}<\len{L_2}$.  Thus, 
\[
\len{\sigma_1 \sigma_2 R_2 }+ \len{\overline{ R}_2  \bar{ \sigma}_2  \bar \sigma_1}<\len{L_2 \sigma_2 R_2 }+ \len{\overline{ R}_2  \bar{\sigma}_2  \overline{L}_2}
\]
by Lemma \ref{lem:plq}.
In other words, this would  reduce the energy  $\len{\mc{X}}$ of $\mc{X}$, and we would be done.

So we are left with considering the case where $\abs{R_1}=0$, or in other words, the case where $c_1=\sigma_1$.  In this case, we must have $c_2=\sigma_2 R_2$.  Lemma \ref{lem:w} says that $R_2$ cannot contain the letter $\sigma_1$.  Thus, if $\abs{R_2}\neq 0$,  the same argument as above  shows that the chord slide of $c_2$ to the right reduces the energy.  If $\abs{R_2}=0$, then we look at $c_3$, etc., until we find the minimum $j$ such that $\abs{R_j}\neq 0$.  Note that a minimum $j$ must exist since $\mc{X}$ is not equivalent to $\mc{S}$.  In this case, we have $c_j=\sigma_j R_j$ with $\abs{R_j}=1$.  Again, Lemma \ref{lem:w} says $R_j$ cannot contain any of the letters $\sigma_i$ for $i<j$ so that the chord slide of $c_j$ to the right reduces the energy. 
\end{proof}

As an immediate consequence, we obtain the 
\begin{theorem}\label{thm:fmain2}
There is an algorithm which from a given marked linear chord diagram $G_0$ provides a sequence of energy  decreasing chord slides
\[
G_0\ra G_1\ra \dotsm \ra G_k, \quad \norm{G_i} < \norm{G_{i-1}}
\]
with $G_k=G_{\mc{S}}$.   In particular, by applying this to $G_0=G_{\varphi(\mc{S})}$ for some $\varphi\in MC_{g,1}$, we obtain an algorithm which factors  mapping classes  into  sequences of chord slides.  
\end{theorem}
\begin{proof}
At every step, the existence of an energy reducing chord slide is provided  by Theorem \ref{thm:fmain2}. That this determines an  algorithm  is guaranteed by the fact that for any linear chord diagram, only finitely many chord slides are possible.  
\end{proof}

\section{Relation to fatgraphs}\label{sect:fatgraphs}

In this final section, we re-articulate  the main result of the previous section into statements about factorizations of mapping classes.  
As the name suggests, the original motivation for the fatgraph Nielsen reduction algorithm comes  from a desire to obtain a factorization into sequences of Whitehead moves on fatgraphs, and we begin this section by briefly outlining how our results on linear chord diagrams transfer to this fatgraph context.  We end the section with a purely mapping class group factorization algorithm where a  mapping class is factored into a product of  certain generators for $MC_{g,1}$.

A \emph{fatgraph} is a finite graph with a cyclic ordering given to the half-edges incident to each vertex.  Given a fatgraph $G$, we can define the boundary cycles of $G$ to be the cyclically ordered sequences of oriented edges where an incoming  edge at a vertex $v$ is followed by the outgoing edge which is next in the cyclic ordering at $v$.  By gluing a disc onto each boundary component of a fatgraph $G$, one obtains a closed orientable surface $\Sigma_G$ and we define the \emph{genus} of the fatgraph to be the genus of the surface $\Sigma_G$.
By a genus $g$ bordered fatgraph, we shall mean a genus $g$ fatgraph with one boundary component with all vertices trivalent except for one which is univalent.  We call the edge incident to the univalent vertex the \emph{tail} $t$ and give it the orientation that it points away from the univalent vertex.  

Fix a genus $g$ surface $\Sigma_{g,1}$ with one boundary component $\partial\Sigma_{g,1}$, and let $q\in \partial\Sigma_{g,1}$ be a point on the boundary distinct from the basepoint $p\in \Sigma_{g,1}$.   By a marking of a genus $g$ bordered fatgraph, we shall mean an isotopy class of embedding of $G$ into $\Sigma_{g,1}$ as a spine such that the image of the end of the tail is the point $q$.  \Poin dual to a marking of a bordered fatgraph $G\hra\Sigma_{g,1}$ is a collection of arcs, which can be chosen to be based at $p$.  In this way, we get a map from the oriented edges of $G$ to $\pi=\pi_1(\Sigma_{g,1},p)$, which we call a $\pi$-marking.  A $\pi$-marking satisfies certain  conditions referred to as  orientation and vertex compatibility, surjectivity, and geometricity, the last of which states that the $\pi$-marking of the tail is the inverse of the element representing the boundary of $\Sigma_{g,1}$,  $\pi(t)=\overline{\partial\Sigma_{g,1}}$. 

We can then define the Ptolemy groupoid as the groupoid with one object for each (isomorphism class of) marked bordered fatgraph and a unique morphism between each pair of objects.  This can be considered as a discrete combinatorial subgroupoid of the fundamental path groupoid of decorated \Teich space given by the 2-skeleten of the dual fatgraph complex for $\Sigma_{g,1}$, but we will not need this perspective here (for more details, see for example \cite{penner, abp}).

By a Whitehead move $W\colon G\ra G'$ on a bordered fatgraph $G$, we shall mean the process of  collapsing  a non-tail edge $e$ of $G$ and expanding the resulting four-valent vertex in the unique opposite way.  It is clear that markings, thus also $\pi$-markings, evolve unambiguously under Whitehead moves.  Note that under obvious conditions, Whitehead moves can be composed.  There are certain sequences of Whitehead moves called the involutivity, commutativity, and pentagon relations  which always take a marked fatgraph to itself.  As a result of decorated \Teich theory or the theory of Strebel differentials that  the Ptolemy groupoid of $\Sigma_{g,1}$ can be equivalently defined as  the groupoid generated by Whitehead moves and with relations given by the involutivity, commutativity, and pentagon relations.  

\subsection{Linear chord diagrams as bordered fatgraphs}

Note that the immersion of a  linear chord diagram $G$ in the plane endows $G$ with the structure of a fatgraph.  Moreover, if $G$ is the linear chord diagram associated to some CG set for $\pi$, then $G$ is a genus $g$ fatgraph with one boundary component.  In fact, we can consider $G$ as a bordered fatgraph by ignoring the bivalent vertex corresponding to the rightmost chord end and ``growing a tail'' on the left:  for example,  by extending the core $[1,4g]\in \bR$ to the  larger interval $[0,4g]$.  %In this way, the rightmost chord end of $G$ 
In this way, one can easily show that a marking of a linear chord diagram $G$ (as described in Section \ref{sect:linfatgraphs}) is equivalent to a $\pi$-marking of $G$ as a bordered fatgraph. 

As observed in \cite{abp}, a chord slide on a marked linear chord diagram can be decomposed as a sequence of two Whitehead moves on the corresponding bordered fatgraph.  In fact, Theorem \ref{thm:fmain2} gives an alternate proof of the fact, first proven in  \cite{bene}, which states  that the chordslide groupoid is a subgroupoid of the Ptolemy groupoid. Moreover, there is an algorithm which takes any bordered fatgraph to its ``nearest linear chord diagram'' via the so-called greedy algorithm  \cite{abp}.  Thus, Theorem \ref{thm:fmain2} can be restated as follows:
\begin{theorem}\label{thm:fmain3} 
Given any ``basepoint'' marked bordered fatgraph $G_0\hra\Sigma_{g,1}$, there is an algorithm which determines a sequence of Whitehead moves
\[
G_0\ra G_1\ra \dotsm \ra G_k=\varphi(G_0)
\]
representing a mapping class $\varphi\in MC_{g,1}$ purely from the action of $\varphi$ on $\pi$. 
\end{theorem}

We note that other algorithms have been previously presented  \cite{mosher,Penner03} which  rely on representing  mapping classes by their action on ideal arcs in $\Sigma_{g,1}$ and resolving intersections of such arcs.

Finally, in \cite{abp}, the so-called chordslide algorithm was discussed and was shown to give an $MC_{g,1}$-equivariant  map from morphisms of the Ptolemy groupoid of $\Sigma_{g,1}$ (i.e., sequences of Whitehead moves) to $MC_{g,1}$.  For any genus $g$, there are only finitely many combinatorial types of Whitehead moves (note that  this shows that $MC_{g,1}$ itself is finitely generated)  and we denote by $\mathfrak{G}_g$ the finite image of all Whitehead moves (i.e., generators of the Ptolemy groupoid) under this map.  
\begin{theorem}
If the elements of a generating set $\mathfrak{S}$ for $MC_{g,1}$ can be explicitly written as products  of elements of $\mathfrak{G}_g$, 
then  Theorem \ref{thm:fmain3} provides an algorithm to decompose any mapping class  $\varphi\in MC_{g,1}$ as a product of elements of $\mathfrak{S}$, purely from the action of $\varphi$ on $\pi$.
\end{theorem}

\bibliographystyle{amsplain}

\end{document}